\documentclass[a4paper, reqno]{amsart}
\usepackage[T1]{fontenc}
\usepackage{color}
\usepackage{amsxtra}
\usepackage{amsfonts}
\usepackage{amssymb}
\usepackage{latexsym}
\usepackage{amsmath}
\usepackage{amsthm}
\usepackage{mathtools}
\usepackage[mathscr]{eucal}
\usepackage{graphicx}
\usepackage[utf8]{inputenc}
\usepackage{subcaption}

\usepackage[ruled]{algorithm2e}
\usepackage{hyperref}

\usepackage[capitalise]{cleveref}

\usepackage[backend=biber]{biblatex}
\renewbibmacro{in:}{%
  \ifentrytype{inproceedings}
    {\printtext{\bibstring{in}\intitlepunct}}
    {}}

\newcommand{\N}{\mathbb{N}}
\newcommand{\R}{\mathbb{R}}

\newcommand{\Z}{\mathbb{Z}}
\newcommand{\Q}{\mathbb{Q}}

\newcommand{\COP}{{\mathcal{COP}}}
\newcommand{\CP}{{\mathcal{CP}}}
\newcommand{\Sym}{{\mathcal{S}}}

\newcommand{\bd}{\mathrm{bd}\,}
\newcommand{\interior}{\mathrm{int}\,}
\newcommand{\boundary}{\mathrm{bd}\,}
\newcommand{\relint}{\mathrm{relint}\,}
\newcommand{\cone}{\mathrm{cone}\,}
\newcommand{\conv}{\mathrm{conv}\,}

\newcommand{\NonN}{\mathcal{N}}
\newcommand{\Ryshkov}{\mathcal{R}}
\newcommand{\glnz}{{\rm GL}_n(\Z)}
\newcommand{\Trace}{{\rm Trace}}
\newcommand{\GL}{\mathrm{GL}}
\newcommand{\V}{\mathcal{V}}

\newcommand{\MinCOP}[1][\COP]{\mathrm{Min}_{#1} \thinspace}
\newcommand{\minCOP}[1][\COP]{\mathrm{min}_{#1} \thinspace}
\newcommand{\RyshkovK}[1][K]{\Ryshkov_{#1}}

\newcommand{\COPK}[1][K]{\COP_{#1}}
\newcommand{\CPK}[1][K]{\CP_{#1}}

\newcommand{\PK}{\mathcal{P}_K}

\newcommand{\extendedVoronoi}{e\mathcal{V}}

\newcommand{\vect}[2]{\big( \mkern-3mu \begin{smallmatrix} {#1} \\ {#2} \end{smallmatrix} \mkern-3mu \big)}

\theoremstyle{plain}
\newtheorem{thm}{Theorem}[section]
\newtheorem{lemma}[thm]{Lemma}
\newtheorem{coro}[thm]{Corollary}

\theoremstyle{definition}
\newtheorem{defi}[thm]{Definition}
\newtheorem{example}[thm]{Example}

\theoremstyle{remark}
\newtheorem{remark}[thm]{Remark}

\author{Alexander Oertel}
\address{A.~Oertel, Universit\"at Rostock, Institute of
  Mathematics, 18051 Rostock, Germany}
\email{alexander.oertel@uni-rostock.de}

\author{Achill Sch\"urmann}
\address{A.~Sch\"urmann, Universit\"at Rostock, Institute of
  Mathematics, 18051 Rostock, Germany}
\email{achill.schuermann@uni-rostock.de}

\date{\today}

\subjclass{11H55, 52B70, 11D09, 11J04, 90C20}
\keywords{Generalized Perfectness, Polyhedral Tesselation, Copositivity, CP-Factorization, Diophantine Approximation, Pell Equation}

\begin{document}

\title[Generalized Perfect Matrices]{Generalized Perfect Matrices}

\begin{abstract}
We generalize Voronoi's theory of perfect quadratic forms to generalized copositive matrices over a full-dimensional closed convex cone~$K$, by introducing a $K$-copositive minimum and perfect $K$-copositive matrices.
We consider a key feature of a given cone, which we call Interior Ryshkov (IR) property.
Under this property the classical theory and its applications generalize nicely and we prove that rationally generated cones possess this IR property.
For contrast, we give a detailed example of a simple cone without the IR property, showing various differences to the classical case.
Moreover, this example yields connections to questions of number theory, in particular to Diophantine approximation and the Pell Equation.
Finally, as an application, we give inner and outer polyhedral approximations for the generalized completely positive cone and a method to find rational certificates for (non-)membership in this cone.
\end{abstract}

\maketitle

\section{Introduction and Overview}

Perfect matrices, respectively perfect quadratic forms, are a classical
topic in number theory, dating back to the 19th century (cf. \cite{kz-1877}). 
Through the reduction theory for positive definite quadratic forms due
to Voronoi~\cite{voronoi-1907}, they became important for different
applications in several mathematical areas, ranging from topics in algebraic
geometry and topology (cf. \cite{MR2208417}, \cite{MR3084439}) 
to numerical algorithms for PDEs (cf. \cite{MR4029820})
and the famous lattice sphere packing problem (cf. \cite{cs-1998}).
An important ingredient in these applications are polyhedral subdivisions of the convex cone of
positive definite real symmetric matrices $\Sym^n_{\succ 0}$.
One of them is given by a locally finite polyhedral set known as the Ryshkov polyhedron,
whose vertices are the perfect matrices (see \cite{CompGeometryForms}).
Dually, another tessellation of $\Sym^n_{\succ 0}$ is obtained from
full-dimensional polyhedral Voronoi cones, each one given by the
minimal vectors of a perfect matrix.
Enumerating all perfect matrices, or Voronoi cones respectively, up to arithmetical equivalence
-- also known as Voronoi's algorithm -- can be thought of as
a graph traversal search among the vertices of the Ryshkov polyhedron.
A direct application of this enumeration is a solution of
the $n$-dimensional lattice sphere packing problem (see \cite{martinet-2003} and also the recent computational breakthrough for $n = 9$ in \cite{dw-2025}).

Over the years Voronoi's reduction theory and the notion of perfect forms or matrices has been
generalized in different directions. Important milestones are the works of
Koecher \cite{koecher-1960,koecher-1961} and Opgenorth \cite{Opgenorth01012001}.
Koecher generalized the corresponding polyhedral tessellations from $\Sym^n_{\succ 0}$ to
tessellations of any self-dual cone.
Opgenorth generalized it
further to the setting of a pair of cones dual to each other.
In the latter theory the polyhedral tessellation into Voronoi cones is constructed in
one convex cone, and the perfect matrices are taken from its dual cone.
As an explicit construction of the abstract theory,
the notion of {\em perfect copositive matrices} in the cone $\COP^n$ of copositive $n\times n$ matrices has been introduced
in~\cite{CertificateAlgo}.
There are several similarities, but also some differences to the
classical theory (see~\cite{PerfectCopositive}).
However, the polyhedral tessellations are quite similar:
Perfect copositive matrices are vertices of a locally finite polyhedral Ryshkov
set in the interior $\interior \COP^n$ and dually to it we obtain a
polyhedral tessellation into Voronoi cones for the cone $\CP^n$ of completely positive matrices.
This tessellation is used in~\cite{CertificateAlgo} to practically obtain certificates
(cp-factorizations) for complete positivity of given matrices.
In fact, a rational cp-factorization can be
obtained algorithmically in the described way whenever it exists.
Calculating these factorizations is a difficult and important problem (cf. \cite{gd-2020, BermanDuerShakedMondererOpenQuestions}).

In this paper we provide a common framework for the classical
theory of perfect matrices and the recently introduced theory of
perfect copositive matrices.
We are in particular interested under which conditions we obtain polyhedral
tessellations with similar properties and if we can use them to obtain
rational certificates for a given matrix to be completely positive in a generalized way.
For this, let us consider a closed convex cone $K\subseteq \R^n$ and the
notion of a {\em generalized completely positive matrix over~$K$},
that is, real symmetric matrices in 
$$\CPK  :=  \cone\{xx^{\sf T} : x \in K\}.$$
Specializing to $K=\R^n$ this gives the classical theory with $\CPK = \Sym^n_{\succcurlyeq 0}$
and for $K=\R^n_{\geq 0}$ we obtain the recently introduced copositive
theory with $\CPK = \CP^n$.
The dual cone of $\CPK$ is $\COPK$,
the cone of {\em generalized copositive matrices over~$K$}.
In it we define {\em perfect $K$-copositive matrices} and a 
generalized Ryshkov set $\RyshkovK$.
We then obtain a tessellation of $ \interior \CPK$ into Voronoi cones associated to the extreme points of $\RyshkovK$.
Note that this set $\RyshkovK$ is in general more complicated than its classical counterparts.
For example, we do not know in general if all its extreme points are also vertices.
See also Remark~\ref{remark:VerticesExtremePoints} on this matter.
Accordingly, we speak of vertices of $\RyshkovK$ only if we know there are only vertices, and of extreme points otherwise.
For more details about $K$-copositive matrices and the tessalation of $\interior \CPK$ we refer to Sections~\ref{sec:GeneralFramework} and~\ref{sec:ApproxToCPK}.

In many ways we have a common behavior for all~$K$:
In the interior of $\COPK$ the vertices of the Ryshkov set $\RyshkovK$
correspond to perfect $K$-copositive matrices (see also Lemma~\ref{lemma:extreme_points_are_vertices}).
In fact, it turns out that the classical and copositive theory generalize
nicely, whenever the Ryshkov set $\RyshkovK$ is in the interior
$\interior \COPK$. We say that $K$ has the
{\em Interior Ryshkov property (IR property)} or that $K$ is \emph{IR} then (see Definition~\ref{def:IR}).
This is, for instance, the case for
all rationally generated cones~$K$ (see Theorem~\ref{thm:RationallyGeneratedImpliesIR}).
In case the cone is IR, we can in principle compute
{\em rational $\CPK$-factorizations} for matrices in $\interior \CPK$.

However, not all cones are IR, as we show by a
$2$-dimensional example cone~$K$ with an irrational generator in Section~\ref{sec:Sqrt2Example}.
Here, we find quite a list of differences compared to the case when the IR property holds:
For instance, the Ryshkov set $\RyshkovK$ may not be a locally finite polyhedron and its extreme points are not necessarily perfect matrices.
These phenomena happen on the boundary, specifically in the set
$\bd \COPK \cap R_K$, which we describe in Sections~\ref{sec:InteriorF} and~\ref{sec:VerticesRyshkovSet}.

It is quite interesting to observe that the IR property does not fail for all cones with
irrational generators. It turns out that this appears to depend on
whether or not the involved irrationals are badly approximable or not
(see Definition~\ref{def:badly_approx}) from within $K$.
In Section~\ref{sec:e_example} we provide an example of another $2$-dimensional~$K$
with an irrational generator that is nonetheless IR.

Our paper is organized as follows.
In Section~\ref{sec:GeneralFramework} we provide the required notions and definitions and describe some results that carry over from the existing theories.
Section~\ref{sec:diophantine_approx} introduces some basics on Diophantine approximation needed for Sections~\ref{sec:RationallyGeneratedCones} and \ref{sec:Sqrt2Example}.
In the first of these two sections we prove that every rationally generated cone is IR and in the second we examine in detail a non-IR cone.
Section~\ref{sec:ApproxToCPK} deals with approximations to $\CPK$ obtained from perfect matrices (or extreme points of $\RyshkovK$) that allow finding certificates for (non-)membership.
They may also be used for approximation of generalized K-copositive optimization problems.
Finally, in the last Section~\ref{sec:OpenQuestions} we describe open questions and further research directions.

\section{A Generalized Framework for Perfectness} \label{sec:GeneralFramework}

\subsection{Basic Definitions and Notation}

Throughout this paper we assume that $K\subseteq \R^n$ is a
full-dimensional closed convex cone.
In particular, with any two vectors $x,y\in K$
also $\lambda x + \mu y \in K$ for all $\lambda,\mu\in\R_{\geq 0}$.

\subsubsection{Generalized Completely Positive and Copositive Matrices}
In the space of real symmetric $n\times n$ matrices $\Sym^n$
we consider the cone of {\em generalized completely positive matrices over $K$}
\begin{align*}
 \CPK &= \cone\{xx^{\sf T} : x \in K\} \\
 &=\left\{\sum_{i=1}^m \alpha_i x_i x_i^{\sf T} : m \in \mathbb{N},\alpha_i \in \mathbb{R}_{\geq 0}, x_i \in K, i = 1, \ldots, m\right\}.
\end{align*}
It is contained in the cone of positive semidefinite matrices
$\Sym^n_{\succcurlyeq 0}$, and equal to this cone if and only if $K=\R^n$.
If $K=\R^n_{\geq 0}$ is the nonnegative orthant, then
$\CPK=\CP^n$ is the classical cone of completely positive matrices.

If we like to certify that a symmetric matrix $A$ is in $\CPK$, then
we can give a {\em $\CPK$-factorization}
$$
A = \sum_{i=1}^m \alpha_i x_i x_i^{\sf T},
$$
where all $x_i \in K$ and $\alpha_i\geq 0$. We speak of
a {\em rational $\CPK$-factorization} if all $\alpha_i$ are
rational and the vectors $x_i$ are all from the set $K\cap \Z^n_{\geq 0}$.

Using the standard scalar product $\langle A, B \rangle = \Trace(AB) = \sum_{i,j=1}^n A_{ij}B_{ij}$
on $\Sym^n$  we get the dual cone of $\CPK$, that is,
{\em the cone of generalized copositive matrices over $K$}
\begin{align*}
 \COPK = (\CPK)^* &= \{Q \in \Sym^n : \langle A, Q \rangle \geq 0 \text{ for all } A \in \CPK \}\\
  &= \{ Q \in \Sym^n :  Q[x] \geq 0 \text{ for all } x\in K \},
\end{align*}
where we use the notation $Q[x] = x^T Q x$.
We always have $\Sym^n_{\succcurlyeq 0} \subseteq \COPK$ and equality holds if and only if $K=\R^n$.
For $K=\R^n_{\geq 0}$, we obtain the classical copositive cone.
We say that matrices in $\COPK$ are \emph{$K$-copositive}.

\subsubsection{Generalized Minima and Perfectness}
Next, we generalize the notions of the arithmetical minimum
(in the classical case $K=\R^n$)
and the copositive minimum (in the case $K=\R^n_{\geq 0}$):
For $Q\in\Sym^n$ we define the {\em $K$-copositive minimum} as
$$
\minCOP[K] Q := \inf \left\{ Q[z]  :  z \in K\cap \Z^n \setminus\{0\} \right\}
.
$$
Having a nonnegative copositive minimum is a characterization of membership in $\COPK$ (see Lemma~\ref{lemma:minCOP_COP_Char} below).

For $Q\in \COPK$ the set of integral vectors attaining the $K$-copositive minimum is
$$
\MinCOP[K] Q := \left\{ z \in K\cap \Z^n \setminus \{0\}  : Q[z] = \minCOP[K] Q \right\}.
$$

A matrix $Q \in \COPK$ is called {\em perfect $K$-copositive}
if it is uniquely determined by its $K$-copositive minimum and the
set of vectors in $\MinCOP[K] Q$ attaining it.
For $K=\R^n$ this gives the classical notion of a perfect matrix and
for $K=\R^n_{\geq 0}$ we call the matrix perfect copositive.
A matrix is perfect $K$-copositive if and only if
the {\em Voronoi cone} of $Q$
$$
{\mathcal{V}} (Q)
:=\cone\{ xx^{\sf T} : x \in \MinCOP[K] Q \} \subseteq\CPK
$$
is full-dimensional, that is, if it has dimension $\dim \Sym^n = \binom{n+1}{2}$.

We define the {\em generalized Ryshkov set over $K$} as
$$
\RyshkovK
=
\big\{ \, Q \in \Sym^n \; : \;  Q[z]\geq 1 \mbox{ for all } z\in K\cap \Z^n \setminus\{0\} \big\}
,
$$
hence as the set of symmetric matrices having a $K$-copositive minimum
of at least~$1$.

\begin{remark}\label{remark:VerticesExtremePoints}
 In the classical theories with $K = \R^n$ and $K = \R_{\geq 0}^n$ the perfect $K$-copositive matrices are precisely the vertices of $\RyshkovK$.
 The Voronoi cone of such a vertex is then the normal cone of $\RyshkovK$ at that vertex and this fact allows for the certificate application mentioned in the introduction.

 For general $K$, the set $\RyshkovK$ may also have extreme points instead of just vertices (see e.g. \cite{Barvinok} for definitions).
 When we generalize the polyhedral approximations to $\CPK$ in Section~\ref{sec:ApproxToCPK}, we thus take care to distinguish between vertices and extreme points.
\end{remark}

\subsubsection{Symmetry}
The classical notion of 
{\em arithmetical equivalence} for two matrices $Q_1, Q_2\in \Sym^n$
refers to the existence of a matrix $U\in \glnz$ with $Q_2= U^{\sf T} Q_1 U$.
Note that $\glnz$ is the linear symmetry group of the set of integral vectors~$\Z^n$.
A corresponding notion of equivalence of matrices for general $K$ can be obtained by replacing $\glnz$ with the linear symmetry group $G$ of $(K \cup (-K)) \cap \Z^n$.
Note that $G$ is the symmetry group of $\RyshkovK$, in the sense of
\begin{equation*}
 \mathrm{Sym} (\RyshkovK) = \{U \in \glnz : U^T \RyshkovK U = \RyshkovK \}.
\end{equation*}
Here, $(-U)^T Q (-U) = U^T Q U$ necessitates using the linear symmetry group of $K \cup (-K)$ instead of just $K$ for $\mathrm{Sym} (\RyshkovK) = G$ to hold.

\begin{example}
 Set $0 \leq k \leq n$.
 The cones $K = \R^k \times \R^{n-k}_{\geq 0}$ induce a family of theories lying naturally between the classical theory and the copositive case.
 The matrix (see \cite[Lemma~2.5]{CertificateAlgo})
 \begin{equation*}
  Q_{A_n} = \begin{pmatrix}
              2 & -1 & 0 &\dots & 0 \\
              -1 & 2 & \ddots & \ddots & \vdots \\
              0 & \ddots & \ddots & \ddots & 0 \\
              \vdots & \ddots & \ddots & 2 & -1 \\
              0 & \dots & 0 & -1 & 2
            \end{pmatrix}
 \end{equation*}
 is perfect copositive and has matching copositive and arithmetical minimum.
 Since $\R^n_{\geq 0} \subseteq \R^{k} \times \R^{n-k}_{\geq 0}$ it is thus also perfect $K$-copositive.
 The symmetry group of $\RyshkovK$, i.e. the subgroup of $\glnz$ giving the notion of generalized arithmetical equivalence, is isomorphic to $\GL_k(\Z) \times (\pm \mathrm{Sym}(n-k))$ for $n-k \neq 1$, where $\pm \mathrm{Sym}(n-k)$ denotes the group of $(n-k)\times(n-k)$ permutation matrices and their negatives.
 Note that for $n-k = 1$ the cone $K = \R^{n-1} \times \R_{\geq 0}$ induces the same $K$-copositive cone as $K = \R^n$, namely $\Sym_{\succcurlyeq 0}^n$, so the symmetry group is $\glnz$ then.
\end{example}

\subsection{Interior Ryshkov Property}
Both the classical arithmetical minimum and the copositive minimum may be used to characterize the membership of a symmetric matrix with the respective cone (cf. \cite[Lemma~1.1]{oertel2025computingcopositiveminimumrepresentatives}).
\begin{lemma}\label{lemma:minCOP_COP_Char}
 Let $K \in \{ \R^n, \R_{\geq 0}^n \}$ and $Q \in \Sym^n$.
 Then
 \begin{itemize}
  \item[(i)]  $\minCOP[K] Q > 0$ if and only if $Q$ is strictly K-copositive,
  \item[(ii)] $\minCOP[K] Q = 0$ if and only if $Q$ is K-copositive, but not strictly so, or
  \item[(iii)] $\minCOP[K] Q = -\infty$ if and only if $Q$ is not K-copositive.
 \end{itemize}
\end{lemma}
The third property extends to all full-dimensional closed convex cones $K$, but in (i) only $\Leftarrow$ does and in (ii) only $\Rightarrow$ remains by Lemma \ref{lemma:strictly_cop_implies_pos_min} below and the denseness of $\Q^n \cap K$ in $K$.

\begin{example}\label{ex:valentin1}
 Set
 \begin{equation*}
  K = \cone \left\{ \begin{pmatrix} \sqrt{2} \\ 1 \end{pmatrix}, \, \begin{pmatrix} 0 \\ 1 \end{pmatrix} \right\} \subseteq \R^2.
 \end{equation*}
 Then
 \begin{equation*}
  Q = \begin{pmatrix}
       -1 & 0 \\
       0 & 2
      \end{pmatrix} \in \boundary \COPK,
 \end{equation*}
 since $Q\big[\vect{\sqrt{2}}{1}\big] = 0$ and for $\alpha, \beta \geq 0$ we have
 \begin{equation*}
  Q\left[ \alpha \begin{pmatrix} \sqrt{2} \\ 1 \end{pmatrix} + \beta  \begin{pmatrix} 0 \\ 1 \end{pmatrix} \right] = 4 \alpha \beta + 2 \beta^2 \geq 0.
 \end{equation*}
 But we also have $\minCOP[K] Q = 1$, because
 \begin{equation*}
  Q \left[ \begin{pmatrix} p \\ q \end{pmatrix} \right] = - p^2 + 2q^2
 \end{equation*}
 is equal to zero if and only if $\frac{p}{q} = \pm \sqrt{2}$.
 Since $Q$ is integral we have thus $Q\left[\vect{p}{q} \right] \geq 1$ for $\vect{p}{q} \in K \cap \Z^n\setminus\{0\}$. The vector $\vect{1}{1}$ is a minimal vector of $Q$.
\end{example}
We investigate this example thoroughly in Section \ref{sec:Sqrt2Example}, since it exhibits a very interesting behavior contrary to the classical theory.

Note that a matrix $Q \in \boundary \COPK$ having always $K$-copositive minimum $0$ is equivalent to the Ryshkov set $\RyshkovK$ being contained in the interior of $\COPK$.
This motivates the following definition.
\begin{defi}[Interior Ryshkov Property]\label{def:IR}
 A closed convex cone $K \subseteq \R^n$ has the \emph{Interior Ryshkov} property (\emph{IR} property), if $Q \in \boundary \COPK$ implies $\minCOP[K] Q = 0$.
 In this case we also say $K$ is \emph{IR}.
\end{defi}
This is the central property guaranteeing a certain niceness of the resulting theory, in the sense that almost everything works out as in the copositive case.
In Section \ref{sec:RationallyGeneratedCones} we prove that rationally generated cones have the IR property and that counterexamples are linked to the boundary of $K$.

\begin{remark}
 Cones that are not IR lead to theories that are excluded from the broader framework of Opgenorth \cite{Opgenorth01012001}:
 Having matrices with positive $K$-copositive minimum in the boundary implies the theory, or specifically the set
 \begin{equation*}
  \left\{zz^T : z\in K \cap \Z^n \setminus \{0\}\right\},
 \end{equation*}
to be what Opgenorth calls \emph{inadmissable}.
Translated to our context, he defines a theory as admissable, if for every sequence $(Q_i)_{i\in\N}$ of strictly $K$-copositive matrices converging to a matrix in $\boundary \COPK$, we also have $(\minCOP[K] Q_i)_{i\in\N}$ converging to $0$ (cf.~\cite[Def.~1.4]{Opgenorth01012001}).
\end{remark}

In the following we give several basic generalizations of results of the $K = \R_{\geq 0}^n$ case and elaborate on the role of the IR property.
We omit the proofs, since they are essentially the same as those in the stated references.
The proof of Theorem~\ref{thm:RationallyGeneratedImpliesIR} given in Section~\ref{sec:RationallyGeneratedCones} is very similar to the one of \cite[Lemma~2.3]{CP_Factorization_First} and may serve as an example how other proofs can be adapted.

The first result concerns matrices in $\interior \COPK$.
It is independent of the IR property and follows from the discreteness of $K \cap \Z^n$. See \cite[Lemma~2.2]{CertificateAlgo}.
\begin{lemma}\label{lemma:strictly_cop_implies_pos_min}
 Let $K$ be a closed convex cone and $Q \in \interior \COPK$.
 Then the $K$-copositive minimum is strictly positive, is attained, and attained by only finitely many vectors.
\end{lemma}

The second result concerns the geometry of the Ryshkov set and follows for IR cones $K$ as in \cite[Lemma~2.4]{CP_Factorization_First}.
A \emph{locally finite polyhedron} is a convex set such that every intersection of the set with a polytope is a polytope as well.
\begin{thm}
 Let $K$ be a full-dimensional closed convex IR cone. Then $\RyshkovK$ is a locally finite polyhedron.
\end{thm}
If $K$ is not IR, then $\RyshkovK$ may fail to be a locally finite polyhedron, as we see in Section~\ref{sec:Sqrt2Example}.
In this case it is not immediately clear how the classical equality of vertices of $\RyshkovK$ and (suitably scaled) perfect $K$-copositive matrices generalizes and how different perfect $K$-copositive matrices can be related to each other.
However, in the interior of $\COPK$ the Ryshkov set then still admits a locally polyhedral structure as the following Lemma implies.
In particular, it shows that all extreme points of $\RyshkovK$ in the interior of $\COPK$ are vertices of $\RyshkovK$, because in this context the Voronoi cone of $Q \in \RyshkovK$ is the normal cone of $\RyshkovK$ at $Q$ (see also Section~\ref{sec:FailureOfIR}).

\begin{lemma}{\label{lemma:extreme_points_are_vertices}}
 Let $K$ be a full-dimensional closed convex cone and let $Q \in \RyshkovK$ such that $Q \in \interior \COPK$ with $\V(Q)$ not full-dimensional.
 Then $Q$ is contained in the relative interior of some line segment contained in $\RyshkovK$.
\end{lemma}
\begin{proof}
 Take $B \in \V(Q)^\perp$, i.e. $B[z] = 0$ for all $z \in \MinCOP[K] Q$.
 We show by contradiction that there is some $\lambda > 0$ such that $Q + \lambda B \in \RyshkovK$.
 Assume that for each $k \in \N$ there is an integral vector $z_k \in K \cap \Z^n \setminus \{0\}$ with
 \begin{equation*}
   \left(Q + \frac{1}{k} B\right)[z_k] < 1.
 \end{equation*}
 By choice of $B$ and $Q \in \RyshkovK$, the sequence $(z_k)_k$ contains infinitely many (pairwise different) integral vectors.
 In particular, $\|z_k\| \to \infty$ for $k \to \infty$.
 Thus,
  \begin{equation*}
  \left( Q + \frac{1}{k} B\right)[z_k] = \| z_k\|^2 \left( Q \left[ \frac{z_k}{\|z_k\|} \right] + \frac{1}{k} B \left[ \frac{z_k}{\|z_k\|} \right] \right)
 \end{equation*}
 cannot remain bounded for $k\to\infty$:
 The continuous map $x \mapsto Q[x]$ is bounded away from zero on $K$ intersected with the unit sphere, because $Q \in \interior \COPK$.
 Thus, the term inside the parentheses is eventually positive.
 We conclude that there is some $\lambda > 0$ such that $Q + \lambda B \in \RyshkovK$ and analogously also some $\mu > 0$ with $Q - \mu B \in \RyshkovK$.
 \end{proof}

\subsection{Embedding the Classical Theory}

Finally, we can embed the classical theory of perfect quadratic forms in this new framework.
In \autocite[Theorem~6.1]{PerfectCopositive} it is proved that every classically perfect matrix admits an arithmetically equivalent perfect copositive one.
We can extend that result from $\COP$ to $\COPK$ for full-dimensional closed convex cones $K$.

\begin{lemma}\label{lemma:basis_in_each_cone}
 Let $K \subseteq \R^n$ be a full-dimensional closed convex cone.
 Then $K \cap \Z^n$ contains a lattice basis of $\Z^n$.
\end{lemma}
\begin{proof}
  This is a slight refinement of a classical argument showing the existence of a lattice basis (of a discrete subgroup of $\R^n$, i.e. of $\Z^n$ in our case).
  For this argument we refer to \cite[Chapter~1.3,~Thm.~2]{GruberLekkerkerker} with the following remark:
  Since $K$ is full-dimensional and contains rational and thus integral points, every time one chooses a vector in the proof given there, one can choose that vector in $K$, leading to a basis contained in $K$.
\end{proof}
\begin{thm}
 Let $K \subseteq \R^n$ be a full-dimensional closed convex cone.
 Let further $Q \in \Sym_{\succ 0}^n$ be classically perfect.
 Then there exists a perfect $K$-copositive matrix $Q_K \in \COPK$ that is arithmetically equivalent to $Q$.
\end{thm}
\begin{proof}
 By \autocite[Theorem~6.1]{PerfectCopositive} we may assume that $Q$ is already perfect copositive and that all integral vectors attaining the arithmetical minimum are in $\R_{\geq 0}^n$.
 Lemma \ref{lemma:basis_in_each_cone} yields a lattice basis $\{u_1,\, \ldots,\, u_n\}$ of $\Z^n$ contained in $K$.
 Set
 \begin{equation*}
  U := \begin{pmatrix}
        | & & | \\
        u_1 & \dots & u_n \\
        | & & |
       \end{pmatrix} \in \GL_n(\Z).
 \end{equation*}
 Then $U \R_{\geq 0}^n \subseteq K$ and $U^T Q U$ is uniquely determined by its $K$-copositive minimum and its minimal vectors in~$K$.
\end{proof}

\section{Preliminaries from Diophantine Approximation}\label{sec:diophantine_approx}
The Interior Ryshkov property appears to be linked to how well the extreme ray generators of $K$ can be approximated by rational vectors.
For this reason we quickly review some needed concepts in this section and refer to the book by Schmidt \cite{Schmidt_Diophantine_Approximation} for details.

\begin{defi}[Badly Approximable]\label{def:badly_approx}
 A real number $\alpha$ is called \emph{badly approximable} if there is a constant $C > 0$, such that
 \begin{equation*}
  \left| \alpha - \frac{p}{q} \right| > \frac{C}{q^2}
 \end{equation*}
 for all $p,\,q \in \Z$, $q > 0$.
\end{defi}
Note that this definition captures a notion that is in contrast to the classical theorems of Dirichlet and Hurwitz, which state that there are infinitely many different rational numbers $\frac{p}{q}$ satisfying the reverse inequality for irrational $\alpha$.
For badly approximable numbers the constants in these theorems cannot be improved arbitrarily.

The approximability of an irrational number $\alpha$ is dictated by its \emph{simple continued fractions} expansion
\begin{equation*}
[a_0; a_1, \, \ldots] = a_0 + \frac{1}{[a_1;a_2, \, \ldots]}.
\end{equation*}
The element $a_i$ is called the \emph{$i$-th partial quotient} and $[a_0;a_1, \ldots, a_i]$ is called the \emph{i-th convergent} of $\alpha$.
They are in a sense the best rational approximations to $\alpha$ and satisfy in particular
\begin{equation}
 \left| \alpha - \frac{p}{q} \right| < \frac{1}{q^2} \label{eq:convergents_good_approx}
 \end{equation}
with $\frac{p}{q}$ being the convergent.
Depending on the parity of $i$ the $i$-th convergent is either larger or smaller than $\alpha$.

The following theorem dates back to Hurwitz \cite{Hurwitz1891} (cf. also \cite[Notes on Chapter~XI]{HardyWright}).
\begin{thm}
 An irrational number $\alpha$ is badly approximable if and only if it has bounded partial quotients.
\end{thm}
\begin{example}
 We have $\sqrt{2} = [1;2,2, \ldots]$, so $\sqrt{2}$ is badly approximable.
 A suitable constant $C$ would be, for example, $\frac{1}{4}$.
\end{example}

Finally, in preparation for Section \ref{sec:RationallyGeneratedCones}, we state Dirichlet's Theorem for simultaneous rational approximation. A proof can be found, for example, in \cite[Thm.~5.2.1]{GeometricAlgorithmsAndCombinatorialOptimization}.
\begin{thm}\label{thm:Dirichlet}
 Let $\alpha_1, \ldots, \alpha_k$ be real numbers and let $\varepsilon$ be a number with $0 < \varepsilon < 1$.
 Then there exist integers $p_1, \ldots, p_k$ and a natural number $q$ with $1 \leq q \leq \varepsilon^{-k}$ such that
 \begin{equation*}
  \left| \alpha_i - \frac{p_i}{q} \right| \leq \frac{\varepsilon}{q} \text{ for all } i = 1, \ldots, k.
 \end{equation*}
\end{thm}

\section{Rationally Generated Cones are IR}
\label{sec:RationallyGeneratedCones}
In this section we prove our main theorem about rationally generated cones.
Exceptionally, we allow cones of lower dimension here as well since that allows for a simpler proof and the possibility to make conclusions about the boundary of $K$ if $K$ is not IR.
As already mentioned, our proof generalizes the one of \cite[Lemma~2.3]{CP_Factorization_First}.
\begin{thm}\label{thm:RationallyGeneratedImpliesIR}
 Let $K$ be a rationally generated closed convex cone.
 Then $K$ is IR, i.e. $\RyshkovK \subseteq \interior \COPK$.
\end{thm}
\begin{proof}
 Suppose there is a matrix $Q \in \boundary \COPK \cap \RyshkovK$.
 Then there is a vector $x \in K \setminus \{0\}$ with $Q[x] = 0$.
 We can assume that $x \in \relint K$, since otherwise we might replace $K$ with the minimal face of $K$ containing $x$.

 Because $K$ is rationally generated, we may write
 \begin{equation*}
  x = \sum_{i = 1}^k \alpha_i v_i \text{ with } \alpha_i > 0, \, k \in \N, \text{ and } v_i \in K \cap \Z^n\setminus \{0\}.
 \end{equation*}
 For all $y \in \langle K \rangle$, that is, for all $y$ in the linear span of $K$, and sufficiently small $\varepsilon >0$, we have
 \begin{align*}
  0 \leq \frac{1}{\varepsilon} Q [ x \pm \varepsilon y] = 0 \pm 2 x^T Q y + \varepsilon Q[y],
 \end{align*}
 since $x \in \relint K$.
 It follows that $2x^T Q y \leq \pm \varepsilon Q[y]$ and thus
 \begin{equation}
  x^T Q y = 0 \text{ for all } y \in \langle K \rangle. \label{eq:xTBy_zero}
 \end{equation}

 For all $\lambda \in \R$ and all $y \in \langle K \rangle$ we also have
 \begin{equation*}
  Q [\lambda x + y] = \lambda^2 Q[x] + 2\lambda x^T Q y + Q[y] = Q[y].
 \end{equation*}

 We apply Dirichlet's Theorem~\ref{thm:Dirichlet} to the vector $\alpha = (\alpha_1,\,\ldots,\, \alpha_k)^T$ of coefficients of the conic combination $x$ with $0 < \varepsilon < 1$ and obtain $p = (p_1, \, \ldots,\, p_k)^T \in \Z^k$ and $q \in \N$.
 Since $\alpha_i > 0$, we may assume that $p_i \geq 0$ and that not all $p_i$ are zero.
 Define
 \begin{equation}
  w := \sum_{i = 1}^k p_i v_i \in K \cap \Z^n \setminus \{0\}. \label{eq:definition_w}
 \end{equation}
 Since $Q \in \RyshkovK$, we have $Q[w] \geq 1$.

 Set $y := qx - w \in \langle K \rangle$ and let $M$ be an upper bound for the values $\| v_i \|_{\infty}$ for $i = 1,\ldots, k$.
 Then
 \begin{equation*}
  \|y\|_{\infty} = \|qx - w \|_{\infty} = \left\|\sum_{i=1}^k (q \alpha_i - p_i) v_i \right\|_\infty \leq k M  \varepsilon.
 \end{equation*}

 Finally, with
 \begin{equation*}
   C := \max_{u \in \langle K \rangle \setminus \{0\}} Q \left[ \frac{u}{\|u\|_\infty}\right]
 \end{equation*}
 we have
 \begin{equation*}
  1 \leq Q [w] = Q[qx - y] = Q[y] \leq C \|y\|_\infty^2 \leq C \cdot (k M \varepsilon)^2
 \end{equation*}
 and thus a contradiction for $\varepsilon$ sufficiently small.
\end{proof}

Our theorem imposes restrictions on matrices in $\RyshkovK$ witnessing the failure of the IR property, which we record in the following corollary.
Recall that an \emph{isotropic} vector $w \in \R^n$ of a symmetric matrix $Q$ is a nonzero vector such that $Q[w] = 0$.

\begin{coro}\label{cor:IsotropicInducesExtremeRay}
 Let $K$ be a closed convex cone and assume that there exists a matrix $Q \in \boundary \COPK \cap \RyshkovK$.
 Then all isotropic vectors of $Q$ in $K$ are in the boundary of~$K$.
\end{coro}
\begin{proof}
 Suppose an isotropic vector $x$ is contained in the interior of $K$.
 Then we could embed $x$ in a rationally generated cone $\tilde{K}$ contained in the interior of $K$.
 Theorem \ref{thm:RationallyGeneratedImpliesIR} now contradicts that $Q \in \boundary \COPK[\tilde{K}] \cap \RyshkovK[\tilde{K}]$.
\end{proof}
\begin{remark}
 The rationality is crucial in the construction of the approximation $w$ to the isotropic vector in Equation~\eqref{eq:definition_w}.
\end{remark}

\section{Squareroot Two Example}\label{sec:Sqrt2Example}
In this section we let
 \begin{equation}
  K = \cone \left\{ \begin{pmatrix} \sqrt{2} \\ 1 \end{pmatrix}, \, \begin{pmatrix} 0 \\ 1 \end{pmatrix} \right\} = \{x \in \R^2 : x_1 \geq 0,\, -x_1 + \sqrt{2} x_2 \geq 0 \}\subseteq \R^2. \label{eq:Cone_K_Def}
 \end{equation}
 \begin{figure}
  \centering
  \includegraphics[width=0.4\textwidth]{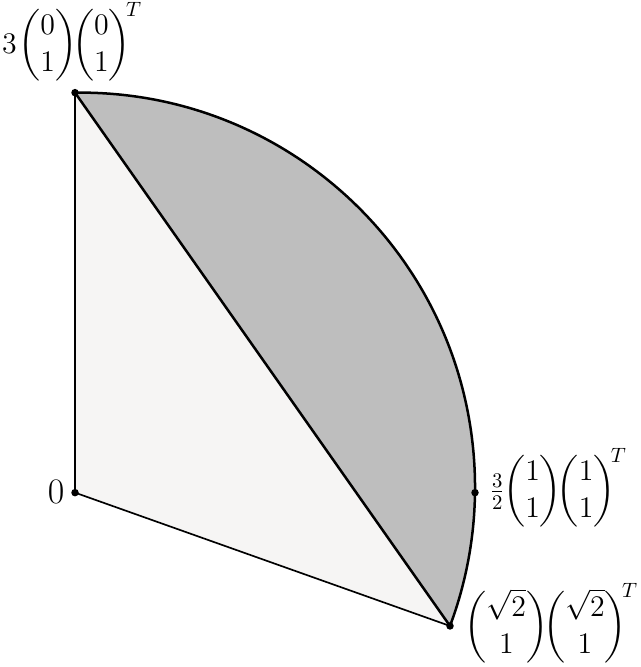}
  \caption{Sketch of $\CPK$ with the dark gray part being the matrices of $\CPK$ with trace $3$.
  }
  \label{figure:cpk}
 \end{figure}
 See Figure~\ref{figure:cpk} for a sketch of the resulting $\CPK$ cone.
 The overarching goal of this section is a throughout discussion of the following theorem.
 \begin{thm}\label{thm:sqrt2_not_IR}
  The cone $K$ defined in \eqref{eq:Cone_K_Def} is not IR.
  Moreover, its Ryshkov set $\RyshkovK$ is not a locally finite polyhedron.
 \end{thm}
 We have seen the validity of the first claim already in Example~\ref{ex:valentin1}.
 In extension,
 we show in this section that we have a $2$-face of $\RyshkovK$ contained in a $2$-face of $\COPK$, so in $\boundary \COPK$.
 Furthermore, we obtain a perfect $K$-copositive matrix with infinitely many minimal vectors corresponding to solutions of the classical (negative) Pell Equation, as well as a non-perfect vertex of $\RyshkovK$.
 This highlights key differences to the copositive and classical cases and also proves the second claim of Theorem~\ref{thm:sqrt2_not_IR}.

 The structure of $K$ tempts one to think that irrational generators lead to non-IR cones.
 However, in Example~\ref{ex:e_Example} we show that irrational generators are not sufficient to rule out the IR property.

 We also note that our example easily generalizes to higher dimensions by a lifting argument:
 Suppose the non-IR cone $K$ is generated by some generators $b_i \in \R^n$ for $i$ in some index set $I$.
 If we lift the generators $b_i$ to $\bar{b}_i = \vect{b_i}{0} \in \R^{n+1}$, then the cone
 \begin{equation*}
  \bar{K} = \cone \left(\{ \bar{b}_i: i \in I \} \cup \{e_{n+1}\}\right).
 \end{equation*}
  is also not IR:
  If $Q \in \boundary \COPK$ with $\minCOP[K] Q = 1$, then
  \begin{equation*}
  \begin{pmatrix}
   Q & 0 \\
   0^T & 1
   \end{pmatrix} \in \Sym^{n+1}
  \end{equation*}
  witnesses the failure of the IR property of $\bar{K}$.
 That shows the existence of ``difficult'' (non-IR) cones in every dimension $n \geq 2$.

 \subsection{Characterizing a 2-Face of the K-Copositive Cone}\label{Sec:2FaceOfKCopCone}
  We start by parameterizing the $2$-face of $\COPK$ on which we can expect matrices with positive $K$-copositive minimum.
  Write
  \begin{equation*}
   Q = \begin{pmatrix}
        a & b \\
        b & c
       \end{pmatrix} \in \Sym^2
  \end{equation*}
  with $a,\,b,\,c \in \R$.
  Corollary \ref{cor:IsotropicInducesExtremeRay} implies that every matrix $Q$ in the boundary of $\COPK$ with nonzero $K$-copositive minimum satisfies $Q\big[\vect{ \sqrt{2}}{1} \big] = 0$.
  This yields
  \begin{equation}
   c = -2a - 2 \sqrt{2} b. \label{eq:c_2}
  \end{equation}

  If $Q \in \COPK$, then Equation~\eqref{eq:c_2} implies $Q \in \boundary \COPK$.
  We have $Q \in \COPK$ if and only if
  \begin{equation*}
   Q \left[ \alpha \begin{pmatrix} \sqrt{2} \\ 1 \end{pmatrix} + \beta \begin{pmatrix} 0 \\ 1 \end{pmatrix} \right] = \alpha^2 (2a + 2 \sqrt{2} b + c) + 2 \alpha \beta (\sqrt{2} b + c) + \beta^2 c \geq 0
  \end{equation*}
  for all $\alpha, \, \beta \geq 0$, so if and only if
  \begin{equation*}
  \tilde{Q} := \begin{pmatrix}
                2a + 2\sqrt{2} b + c & \sqrt{2} b + c \\
                \sqrt{2} b + c & c
               \end{pmatrix}
               \in \COP^2 = \NonN^2 \cup \Sym_{\succcurlyeq 0}^2,
  \end{equation*}
  where $\NonN^n$ denotes the symmetric $n\times n$ matrices with nonnegative entries (see \cite{CopositiveAndCompletelyPositiveMatrices}).
  Substituting $c$ from Equation~\eqref{eq:c_2}, we see that $\tilde{Q} \in \Sym_{\succcurlyeq 0}^2$ implies $\tilde{Q} \in \NonN^2$.
  Thus, we have
  \begin{equation*}
   F := \left\{ \begin{pmatrix}
                 a & b \\
                 b & -2a - 2\sqrt{2} b
                \end{pmatrix} \in \Sym^2
                :
                -2a - 2\sqrt{2} b \geq 0, \, -2a - \sqrt{2} b \geq 0
                \right\}
  \end{equation*}
  as the relevant $2$-face of $\COPK$.

  \subsection{The Relative Boundary of the 2-Face}\label{sec:RelativeBoundaryOfF}
  Matrices in the relative boundary of $F$ have zero as their $K$-copositive minimum.

  If we set $-2a - 2 \sqrt{2} b = 0$, we obtain the ray generated by
  \begin{equation*}
   \begin{pmatrix}
    -1 & \frac{\sqrt{2}}{2} \\
    \frac{\sqrt{2}}{2} & 0
   \end{pmatrix}.
  \end{equation*}
  On this ray the matrices evaluate to zero on $\vect{0}{1}$.
  Hence, the $K$-copositive minimum is then zero as well.

  If we set $-2a - \sqrt{2} b = 0$, we obtain the ray generated by
  \begin{equation*}
   Q = \begin{pmatrix}
        1 & -\sqrt{2} \\
        -\sqrt{2} & 2
       \end{pmatrix} = \begin{pmatrix} 1 \\ - \sqrt{2} \end{pmatrix} \begin{pmatrix} 1 \\ - \sqrt{2} \end{pmatrix}^T.
  \end{equation*}

  We show $\minCOP[K] Q = 0$ next.
  For every second convergent $\frac{p}{q}$ of $\sqrt{2}$ we have $\vect{p}{q} \in K \cap \Z^2\setminus\{0\}$ and by Equation~\eqref{eq:convergents_good_approx}
  \begin{equation*}
   Q \left[ \begin{pmatrix}
             \frac{p}{q} \\ 1
            \end{pmatrix}
             \right] = \left(\sqrt{2} - \frac{p}{q} \right)^2 \leq \left( \frac{1}{q^2} \right)^2.
  \end{equation*}
  Then
  \begin{equation*}
   Q \left[ \begin{pmatrix} p \\ q \end{pmatrix} \right] \leq \frac{1}{q^2}.
  \end{equation*}
  Since $q \to \infty$ for $i \to \infty$, $\minCOP[K] Q = 0$ follows.
  Note that this infimum is not attained.

  \subsection{The Relative Interior of \texorpdfstring{$F$}{F}}\label{sec:InteriorF}
  We consider the relative interior of $F$ now, so we assume the strict inequalities
  \begin{align}
   -2a - 2\sqrt{2} b &> 0 \label{ineq:c_greater_zero} \text{ and } \\
   -2a - \sqrt{2} b &> 0 \label{ineq:second_F_ineq}.
  \end{align}
  We show that the corresponding matrices $Q \in F$ have $K$-copositive minimum strictly greater than zero using the fact that $\sqrt{2}$ is badly approximable.
  Take a constant $C > 0$ such that
  \begin{equation*}
   \left| \sqrt{2} - \frac{p}{q} \right| > \frac{C}{q^2}
  \end{equation*}
  for all $p,q \in \Z$ with $q > 0$.
  We distinguish between the signs of $a$ and $b$.
  Since the main argument is conceptually identical in all cases we consider only the case where $a < 0$ and $b > 0$, for which we have some extra work to do.

  Define
  \begin{equation*}
   f(t) := Q \left[ \begin{pmatrix} t \\ 1 \end{pmatrix} \right] = a \left( t + \frac{b}{a} \right)^2  - \frac{b^2}{a} -2a -2 \sqrt{2} b
   \end{equation*}
   for $t \in [0, \sqrt{2}]$.
   Inequality~\eqref{ineq:c_greater_zero} implies $f(0) > 0$ and also $\left| \frac{b}{a} \right| < \frac{\sqrt{2}}{2}$.
   The graph of $f$ is a downwards opening parabola with its apex having $x$-coordinate $\left| \frac{b}{a} \right|$, see Figure~\ref{fig:graph_f}.
   Note that $f(\sqrt{2}) = 0$.

   \begin{figure}
    \centering
    \includegraphics[width=0.7\textwidth]{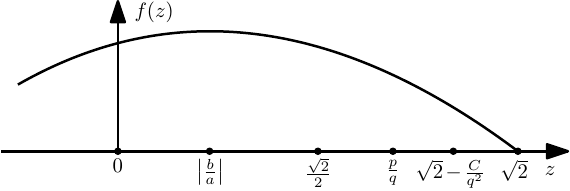}
    \caption{Graph of $f$.}
    \label{fig:graph_f}
   \end{figure}

   Let $\vect{p}{q} \in K \cap \Z^2\setminus\{0\}$.
   Consider first the case $0 \leq \frac{p}{q} < \frac{\sqrt{2}}{2}$.
   Since $f$ is continuous and strictly greater than zero on $[0, \frac{\sqrt{2}}{2}]$, there is a constant $M > 0$ with $f(t) \geq M$ for all $t \in [0, \frac{\sqrt{2}}{2}]$.
   Thus
   \begin{equation*}
    Q \left[ \begin{pmatrix} p \\ q \end{pmatrix} \right] \geq M q^2.
   \end{equation*}

   Now consider $\frac{\sqrt{2}}{2} < \frac{p}{q} < \sqrt{2}$.
   We have
   \begin{equation*}
    \left| \frac{b}{a} \right| < \frac{\sqrt{2}}{2} < \frac{p}{q} < \sqrt{2} - \frac{C}{q^2} < \sqrt{2}
   \end{equation*}
   and thus
   \begin{align*}
    Q\left[ \begin{pmatrix} \frac{p}{q} \\ 1 \end{pmatrix} \right] = f\left( \frac{p}{q} \right) &> f\left( \sqrt{2} - \frac{C}{q^2} \right) \\
    &= \left(-2\sqrt{2} a - 2 b \right) \frac{C}{q^2} + a \left( \frac{C}{q^2} \right)^2.
   \end{align*}
   This implies
   \begin{equation*}
    Q \left[ \begin{pmatrix} p \\ q \end{pmatrix} \right] > \left(-2 \sqrt{2} a - 2 b \right) C + a \frac{C^2}{q^2}.
   \end{equation*}
   Note that the term in parentheses is positive by Inequality~\eqref{ineq:second_F_ineq}.
   As a consequence, for sufficiently large $q$, say $q > N$ for a number $N \in \N$, the value $Q\left[ \vect{p}{q} \right]$ is bounded away from zero.

   Finally, for $q \leq N$ we find $0 \leq p \leq \sqrt{2} q \leq \sqrt{2} N$.
   There are in particular only finitely many such integer points $\vect{p}{q}$.
   For each of them we have $f \big(\frac{p}{q}\big) > 0$, since $f > 0$ on $[0, \sqrt{2})$.
   So $Q\left[ \vect{p}{q} \right]$ is bounded away from zero in this case as well.
  \subsection{Vertices of the Ryshkov Polyhedral Set}\label{sec:VerticesRyshkovSet}
  We describe the geometry of $\RyshkovK \cap F$ next, in particular its vertices. See Figure~\ref{fig:ryshkov_vertices} for a sketch.
  Note that these vertices are also vertices of $\RyshkovK$, as we validate in Section~\ref{sec:ClassicalVoronoiConesAreNotSuff}.
  The first vertex is
  \begin{equation*}
   Q_1 = \begin{pmatrix}
        -1 & 0 \\
        0 & 2
       \end{pmatrix}.
  \end{equation*}
  It is the same matrix as in Example \ref{ex:valentin1}.
  The minimal vectors of this matrix correspond to the positive solutions of the negative Pell Equation $p^2 - 2q^2 = -1$.
  For more details about this and more general equations we refer to \cite[Sec.~57]{NagellNumberTheory} and \cite[Sec.~4.3]{TopologyOfNumbers}.
  There are in particular infinitely many different minimal vectors, the first three being
  \begin{equation*}
   \begin{pmatrix}
   1 \\
   1
   \end{pmatrix}, \,
   \begin{pmatrix}
    7 \\
    5
   \end{pmatrix},\,
   \begin{pmatrix}
    41 \\
    29
   \end{pmatrix}.
  \end{equation*}
  Note that an infinite number of minimal vectors for a perfect matrix is a new phenomenon that is not possible in the classical and in the copositive theory.
  It is also up to scalar multiples the only matrix in $\COPK$ with this property.
  Also note that our definition of perfectness differs from the one by Opgenorth \cite{Opgenorth01012001} in that we also allow matrices in the boundary of $\COPK$ to be perfect as long as their Voronoi cones are full-dimensional.

  There are two $1$-faces of $\RyshkovK \cap F$ containing $Q_1$.
  Finding them requires more precise bounds for the approximability of $\sqrt{2}$.
  One can prove the bound
  \begin{equation*}
   \left| \sqrt{2} - \frac{p}{q} \right| \geq \frac{\sqrt{2}}{4q^2}
  \end{equation*}
  for every second convergent $\frac{p}{q}$ of $\sqrt{2}$ (cf. \cite[Sec.~1,~Thm.~5F]{Schmidt_Diophantine_Approximation}).
  Using this bound one can show that there is only one $1$-face $F'$ of $\RyshkovK \cap F$, defined by $Q\big[ \vect{p}{q} \big] = 1$ for $\vect{p}{q} \in K \cap \Z^2 \setminus \{0\}$ for all $Q \in F'$. It is the one with $\vect{p}{q} = \vect{1}{1}$.
  It follows that the missing $1$-face is the ray
  \begin{equation*}
   \left\{ Q_1 + \lambda \begin{pmatrix}
                          1 & -\sqrt{2} \\
                          -\sqrt{2} & 2
                         \end{pmatrix} : \lambda \geq 0 \right\}
  \end{equation*}
  parallel to the corresponding edge of $F$.
  There is no further vertex on this ray.
  Indeed, since the generator of the ray is $K$-copositive,
  \begin{equation}
   \left( Q_1 + \lambda \begin{pmatrix} 1 & -\sqrt{2} \\
                          -\sqrt{2} & 2
                         \end{pmatrix} \right) \left[ \begin{pmatrix} p \\ q \end{pmatrix} \right] \geq Q_1 \left[ \begin{pmatrix} p \\ q \end{pmatrix} \right]  \geq 1 \label{eq:no_further_vertex}
  \end{equation}
  for $\vect{p}{q} \in K \cap \Z^2$ and $\lambda \geq 0$.
  The existence of such a ray that originates from a limiting process and not from attained minimal vectors is also a new phenomenon.
  Also note that the symmetric matrices on the ray (besides $Q_1$) have no minimal vectors, which is also not possible in the classical and copositive theory.

  The second $1$-face is
  \begin{equation*}
   \left\{ Q_1 + \lambda \begin{pmatrix} 2 - 2\sqrt{2} &  1 \\ 1 & 2 \sqrt{2} - 4 \end{pmatrix} : \lambda \in \left[0, \frac{\sqrt{2}}{4} + \frac{1}{2} \right] \right\}.
  \end{equation*}
  For matrices in this edge we have $Q\big[\vect{1}{1}\big] = 1$.

  Its second vertex is
  \begin{equation*}
   Q_2 = \begin{pmatrix}
        -1 - \frac{\sqrt{2}}{2} & \frac{\sqrt{2}}{4} + \frac{1}{2} \\
        \frac{\sqrt{2}}{4} + \frac{1}{2} & 1
       \end{pmatrix}
  \end{equation*}
  Its minimal vectors are $\vect{1}{1}$ and $\vect{0}{1}$.
  Although it is a vertex of $\RyshkovK$, it is not perfect $K$-copositive.
  This is also a new phenomenon.

  Finally, we have the ray
  \begin{equation*}
  \left\{ Q_2 + \lambda \begin{pmatrix}
                          -1 & \frac{\sqrt{2}}{2} \\
                          \frac{\sqrt{2}}{2} & 0
                        \end{pmatrix} : \lambda \geq 0 \right\}
  \end{equation*}
  parallel to the corresponding edge of $F$.
  Matrices in this edge satisfy $Q\big[\vect{0}{1}\big] = 1$ and that is their only minimal vector (with the exception of $Q_2$ having a second minimal vector).
  By the same argument as in Equation~\eqref{eq:no_further_vertex} we find no further vertices on this ray.
  \begin{figure}
   \centering
   \includegraphics[width=0.8\textwidth]{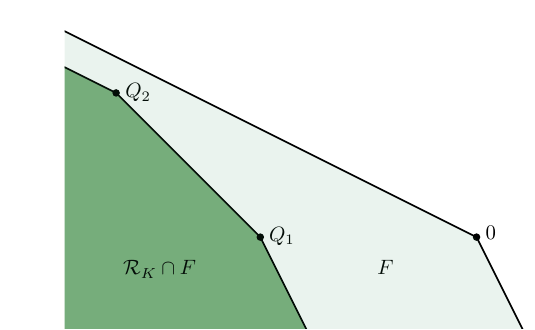}
   \caption{Sketch of $\RyshkovK \cap F$.}
   \label{fig:ryshkov_vertices}
  \end{figure}

\subsection{Voronoi Cones of the Vertices} \label{sec:ClassicalVoronoiConesAreNotSuff}
We now show that the so far used notion of Voronoi cones is not sufficient to tesselate the interior of $\CPK$ (as we can in the classical theories).
Take
\begin{equation*}
 A = \begin{pmatrix} \sqrt{2} \\ 1 \end{pmatrix} \begin{pmatrix} \sqrt{2} \\ 1 \end{pmatrix}^T
 + \begin{pmatrix} 1 \\ 1 \end{pmatrix} \begin{pmatrix} 1 \\ 1 \end{pmatrix}^T
 = \begin{pmatrix} 3 & 1 + \sqrt{2} \\ 1 + \sqrt{2} & 2 \end{pmatrix} \in \interior \CPK.
\end{equation*}
The minimum of $Q \mapsto \langle A, Q \rangle$ over $\RyshkovK$ is attained at the vertices $Q_1$ and $Q_2$ of Section~\ref{sec:VerticesRyshkovSet}.
The Voronoi cone of $Q_2$ is
\begin{equation*}
 \cone \left\{ \begin{pmatrix} 1 \\ 1 \end{pmatrix} \begin{pmatrix} 1 \\ 1 \end{pmatrix}^T,\, \begin{pmatrix} 0 \\ 1 \end{pmatrix} \begin{pmatrix} 0 \\ 1 \end{pmatrix}^T \right\} = \left\{ \begin{pmatrix} a & a \\ a & b \end{pmatrix} : b \geq a \geq 0 \right\}
\end{equation*}
and does not contain $A$.
Similarly, one can show that the infinitely generated Voronoi cone of $Q_1$ does not contain $A$ either.
(It is contained in the boundary of the cone, which is not closed since $\vect{\sqrt{2}}{1} \vect{\sqrt{2}}{1}^T$ is a limit direction, but $\vect{\sqrt{2}}{1}$ is not a minimal vector.)
This shows another difference to the classical case.
In particular, there are matrices in the interior of $\CPK$ not contained in any Voronoi cone.

Note that adding the for $Q_1$ and $Q_2$ common isotropic vector $\vect{\sqrt{2}}{1}$ to the list of minimal vectors and taking the conic hull over the ensuing rank $1$ matrices gives the normal cone of $\RyshkovK$ at $Q_1$, respectively, $Q_2$ again.
In both of these cones we can find $A$, so we can manage to find tessellations by extending the Voronoi cones (see Section~\ref{sec:FailureOfIR}).
Both of these normal cones are also full-dimensional, showing that $Q_1$ and $Q_2$ are indeed vertices of $\RyshkovK$.

\subsection{A Contrary Example}\label{sec:e_example}
To complement the so far given example we consider a at first glance similar example displaying a different behavior.
This shows that irrational generators are not sufficient to destroy the Interior Ryshkov property.

Euler proved the following simple continued fractions expansion, cf. e.g. \cite[Sec.~1.3.2]{FinchMathematicalConstants}.
\begin{lemma}
 For Euler's number $e$ we have
 \begin{equation*}
  e = [2;\, a_1,\, a_2,\, \ldots]
 \end{equation*}
 with
 \begin{equation*}
  a_i = \begin{cases}
         1 & i \equiv 0,\, 1 \mod 3, \\
         2 \cdot \frac{i + 1}{3} & i \equiv 2\phantom{,\, 1} \mod 3
        \end{cases}
 \end{equation*}
 for $i\in\N$.
\end{lemma}

\begin{remark}
 It follows that $e$ is not badly approximable.
 Since the partial quotients of both odd and even order are unbounded, the proof of \cite[Sec.~1,~Thm.~5F]{Schmidt_Diophantine_Approximation} yields that
 \begin{equation*}
  \left| e - \frac{p}{q} \right| < \frac{\varepsilon}{q^2}
 \end{equation*}
 can be achieved by convergents smaller than $e$ for all $\varepsilon > 0$.
\end{remark}
 \begin{example}\label{ex:e_Example}
  Set
 \begin{equation*}
  K = \cone \left\{ \begin{pmatrix} e \\ 1 \end{pmatrix}, \, \begin{pmatrix} 0 \\ 1 \end{pmatrix} \right\} = \{ x \in \R^2 : x_1 \geq 0,\, -x_1 + e x_2 \geq 0 \} \subseteq \R^2.
 \end{equation*}
 We show that $K$ is IR, i.e. that all matrices in the boundary of $\COPK$ have $K$-copositive minimum zero.

  All matrices which could stand in opposition against the IR property have $\vect{e}{1}$ as an isotropic vector, which follows from Corollary~\ref{cor:IsotropicInducesExtremeRay}.

 Analogously to Section~\ref{Sec:2FaceOfKCopCone} we find that the face $F$ of $\COPK$ of such matrices is given by
 \begin{equation*}
  F = \left\{ \begin{pmatrix}
                 a & b \\
                 b & -e^2 a - 2e b
                \end{pmatrix} \in \Sym^2
                :
                -e^2a - 2 e b \geq 0, \, -e^2a - eb \geq 0 \right\}.
 \end{equation*}
 As in Section~\ref{sec:RelativeBoundaryOfF} we find that the matrices in the relative boundary of $F$ have $K$-copositive minimum zero.

 To show that this is the case for the matrices $Q$ in the interior of $F$ as well, let $\varepsilon > 0$ and assume $\varepsilon < e$.
 Let further $\frac{p}{q}$ be a convergent of $e$ smaller than $e$ with
 \begin{equation*}
  e - \frac{p}{q} < \frac{\varepsilon}{q^2}.
 \end{equation*}
 Here we prove for simplicity only the case $a < 0,\, b \leq 0$.
 We have
 \begin{equation*}
  -\frac{b}{a} \leq 0 \leq e - \frac{\varepsilon}{q^2} \leq \frac{p}{q} \leq e.
 \end{equation*}
 Setting
 \begin{equation*}
  f(t) := Q \left[ \begin{pmatrix} t \\ 1 \end{pmatrix} \right] = a \left( t + \frac{b}{a} \right)^2 - \frac{b^2}{a} - e^2 a - 2 e b,
 \end{equation*}
 the graph of $f$ is a downwards opening parabola.
 Since $f(0) = -e^2a - 2eb > 0$ and $f$ attains its maximum at $-\frac{b}{a}$, we have
 \begin{align*}
  f\left(\frac{p}{q} \right) &\leq f\left(e - \frac{\varepsilon}{q^2}\right) \\
  &= (-2ae -2b) \frac{\varepsilon}{q^2}  + a \frac{\varepsilon^2}{q^4}.
 \end{align*}
 So
 \begin{equation*}
  Q \left[ \begin{pmatrix} p \\ q \end{pmatrix} \right] \leq (-2ae - 2b) \varepsilon + a \frac{\varepsilon^2}{q^2}
 \end{equation*}
 and thus $\minCOP[K] Q = 0$.
 \end{example}

\section{Approximations To \texorpdfstring{$\CPK$}{CPK}}\label{sec:ApproxToCPK}
In this section we describe inner and outer polyhedral approximations to $\CPK$ using perfect $K$-copositive matrices, or, more generally, the extreme points of $\RyshkovK$ if $K$ is not IR.
In practice, if $K$ is IR, we can perform a graph traversal search on the vertices of $\RyshkovK$ to find suitable perfect $K$-copositive matrices:
For $K = \R^n$ this is the classical Voronoi algorithm and for $K = \R^n_{\geq 0}$ this is the certificate algorithm of \cite{CertificateAlgo}.
See that reference specifically for details on the traversal and how the certificates emerge from the approximations.

\subsection{Tesselation of the Interior of \texorpdfstring{$\CPK$}{CPK} into Voronoi Cones}
We now deal with an inner approximation to $\CPK$, which enables certifying membership in $\CPK$.
We can tesselate $\interior \CPK$ into (extended) Voronoi cones belonging to the extreme points of $\RyshkovK$.
In the case that $K$ is IR this is completely analogue to the classical cases, where the perfect $K$-copositive matrices correspond to the vertices of $\RyshkovK$.
If $K$ fails to be IR, $\RyshkovK$ may fail to be locally polyhedral and we have to consider the more general extreme points of $\RyshkovK$.

\subsubsection{Cones Satisfying the IR Property}
First we suppose that $K$ satisfies the IR property.
Recall that the Voronoi cone of a perfect $K$-copositive matrix $Q$ is defined as
\begin{equation*}
 \V (Q) = \cone \{ v v^T : v \in \MinCOP[K]Q \}.
\end{equation*}

As the proof of \cite[Thm~1.1]{CP_Factorization_First} generalizes if $K$ is IR, we find that every rational matrix $A \in \interior \CPK$ possesses a rational $\CPK$-factorization.
Algorithmically, for $A \in \interior \CPK$ one minimizes $Q \mapsto \langle A, Q \rangle$ over $\RyshkovK$, or, more specifically, over the set
\begin{equation*}
 \mathcal{P}_K(A, \lambda) = \RyshkovK \cap \{B \in \Sym^n : \langle A, B \rangle \leq \lambda \}
\end{equation*}
 for some sufficiently large value $\lambda > 0$.
In the case that $K$ is IR the set $\PK(A, \lambda)$ is actually a polytope (see \cite[Lemma~2.4]{CP_Factorization_First}, the proof generalizes again if $K$ is IR).
 The minimum is thus attained at some vertex $Q^*$ of $\RyshkovK$ and $\V(Q^*)$ corresponds to the normal cone of $\RyshkovK$ at $Q^*$.

 This implies that the algorithm in \cite{CertificateAlgo} for calculating rational $\CP$-factorizations can be adapted by generalizing the necessary copositive minimum calculations to $K$-copositive minimum calculations.

 \subsubsection{Cones without the IR Property}\label{sec:FailureOfIR}
 If K does not have the IR property, the situation seems to be severely more complicated, as seen in
 Section~\ref{sec:ClassicalVoronoiConesAreNotSuff}.
 In general, we can still tesselate $\interior \CPK$ into \emph{extended Voronoi cones}:
 For an extreme point $Q$ of the Ryshkov set $\RyshkovK$, we define
 \begin{equation*}
  \extendedVoronoi (Q) := \cone \left( \{ vv^T : v \in \MinCOP[K] Q \} \cup \{ w w^T : w \in K,\, Q[w] = 0\} \right).
 \end{equation*}
 It is the normal cone of $Q$ with respect to $\RyshkovK$.
 Let $A \in \interior \CPK$.
 For some sufficiently large $\lambda$, minimize $Q \mapsto \langle Q, A \rangle$ over $\PK(A, \lambda)$ and let $Q^*$ be an optimal solution.
 This is possible even if $K$ is not IR since the proof of \cite[Lemma~2.4]{CP_Factorization_First} yields that the set $\PK(A, \lambda)$ is still a compact convex set.
 This shows that $A$ is contained in the normal cone of $Q^*$, i.e. in some extended Voronoi cone.

 Proving that $\extendedVoronoi (Q^*)$ is the normal cone can be done roughly along the lines of \cite[Thm~1.1]{CP_Factorization_First} by some technical arguments:
 The adjunction of the isotropic elements ensures that $\extendedVoronoi(Q^*)$ is closed.
 If $A \notin \extendedVoronoi(Q^*)$, we could strictly separate $A$ and $\extendedVoronoi(Q^*)$.
 From this separation we would find a matrix $Q'$ that is better than $Q^*$ (with respect to the minimization problem) and also contained in $\RyshkovK$, contradicting the optimality of $Q^*$.
 Showing the latter fact requires one to consider the vectors $v \in K \cap \Z^n \setminus \{0\}$ ``close'' to isotropic vectors carefully for the condition $Q'[v] \geq 1$.

 This allows in principle to calculate certificates.
 However, the possible lack of finiteness of $\MinCOP[K] Q^*$ and the irrationality seem to imply some practical hurdles.
 Moreover, if $Q^*$ is isotropic on a face of $K$ of dimension at least two, then $\extendedVoronoi(Q^*)$ is also not finitely generated.
 So only relatively tame cones like those in Section~\ref{sec:Sqrt2Example} are eligible in general here.
 Nonetheless, it is still possible that even for difficult cones concrete certificates rely only classical Voronoi cones.

\subsection{An Outer Approximation to \texorpdfstring{$\CPK$}{CPK}}
In this section we give an outer approximation to $\CPK$ using the extreme points and extreme rays of $\RyshkovK$.
This is analogue to the polyhedral subdivision of $\Sym_{\succ 0}^n$ mentioned in the introduction and to the one given in \cite[Theorem~2.3]{CertificateAlgo} in the copositive case.
They allow for certification of non-membership in $\CPK$.

\begin{thm}\label{thm:OuterApprox}
 Let $K$ be a full-dimensional closed convex cone.
 Then
 \begin{equation*}
  \CPK = \left\{ A \in \Sym^n : \langle A, Q \rangle \geq 0 \text{ for all extreme points and extreme rays } Q \text{ of } \RyshkovK \right\}.
 \end{equation*}
\end{thm}
\begin{proof}
 We have $\interior \COPK \subseteq \cone \RyshkovK \subseteq \COPK$ and thus by dualizing $\CPK = (\cone \RyshkovK)^*$.
 By \cite[Theorem~18.5]{ConvexAnalysisRockafellar} we have
 \begin{equation*}
  \RyshkovK = \conv \{\text{extreme points of } \RyshkovK\} + \cone \{\text{extreme rays of } \RyshkovK\}
 \end{equation*}
 and through a standard calculation for $(\cone \RyshkovK)^*$ the conclusion follows.
\end{proof}

If $K$ is IR, the extreme points are the perfect $K$-copositive matrices.
They can be found iteratively by a graph traversal search of the vertices of $\RyshkovK$ as in the classical Voronoi algorithm.

In the classical case $K = \R^n$ the Ryshkov polyhedron has no extreme rays (see e.g. \cite[Thm~7.2.1]{martinet-2003}).
Here one obtains
\begin{align*}
 \Sym_{\succcurlyeq 0}^n = \{Q \in \Sym^n : \langle Q, P \rangle \geq 0 \text{ for all perfect } P\}.
\end{align*}
The copositive case $K = \R^n_{\geq 0}$ has been treated in \cite{PerfectCopositive}, where the extreme rays of $\RyshkovK$ have been obtained.
They coincide with the extreme rays of $\NonN^n$, the cone of real nonnegative symmetric $n\times n$ matrices.
This yields the characterization
\begin{equation*}
 \CPK = \left\{ A \in \NonN^n : \langle A, P \rangle \geq 0 \text{ for all perfect copositive } P \right\}.
\end{equation*}

\section{Open Questions}\label{sec:OpenQuestions}
\subsection{Practical Computations}
While it is clear that for cones satisfying the IR property we can apply the graph traversal algorithm on the vertex graph of $\RyshkovK$ to find $\CPK$-certificates in general, it is not clear how to solve the necessary copositive minimum problems, see \cite{CertificateAlgo, oertel2025computingcopositiveminimumrepresentatives} for strategies for $K = \R^n_{\geq 0}$.
They are needed to calculate the Voronoi cones and to find contiguous perfect $K$-copositive matrices.
To practically compute certificates the approaches of \cite{CertificateAlgo, oertel2025computingcopositiveminimumrepresentatives} would need to be generalized.
For finitely and rationally generated cones this seems to be feasible.

For cones not satisfying the IR property it is additionally not clear at all how to perform the graph traversal search.
In particular, that is the case if we have extreme points that are not vertices, or a non-discrete vertex set (see Section~\ref{sec:OpenGeometry}).
Completely new ideas might be needed here.

For sufficiently nice non-IR cones $K$, where we can in principle do a graph traversal search, we could still encounter vertices of $\RyshkovK$ with infinitely many minimal vectors.
Here new conceptual ideas are needed on how to describe a potentially infinite number of minimal vectors and how to decide which extreme ray of the corresponding (extended) Voronoi cone to choose to find the next vertex.
The example of Section~\ref{sec:Sqrt2Example} suggests that we could treat the isotropic vector as a ``virtual'' minimal vector, but we do not currently know how to make that precise.
It seems even more difficult to generalize this idea to more complicated cones.

\subsection{Classification of IR Cones}
In dimension 2 the proofs of the example of Section \ref{sec:Sqrt2Example} and Example \ref{ex:e_Example} show that the IR property depends essentially only on the approximability of the extreme ray generators of $K$ from within $K$.
Since IR cones are much nicer to handle algorithmically, it would be very nice to have further results on conditions related to the IR property besides Theorem \ref{thm:RationallyGeneratedImpliesIR}.

\subsection{Geometry of the Ryshkov Set}\label{sec:OpenGeometry}
In the outer approximations to $\CPK$ of Theorem \ref{thm:OuterApprox} the extreme rays of $\RyshkovK$ play a central role.
For $K = \R^n_{\geq 0}$ these rays are very simple and it seems feasible that the extreme rays of $\RyshkovK$ can be classified if $K$ is simple enough, e.g. finitely generated.

Besides that the geometry of $\RyshkovK$ is reasonably well understood if $K$ has the IR property since $\RyshkovK$ is a locally finite polyhedron then.
Otherwise we know a lot less.
A basic question is if every extreme point is also a vertex and if so, if the vertex set is discrete.
A promising approach to settle these questions is to consider the non-polyhedral ice cream cone
\begin{equation*}
 C_{n+1} = \left\{ \vect{x}{t} : x \in \R^n, t \in \R, \|x\| \leq t \right\},
\end{equation*}
also known as the Lorentz cone.
A nice characterization of $\COPK[C_{n+1}]$ of Loewy and Schneider \cite[Lemma~2.2]{LOEWY1975375} allows one to consider concrete examples.
By a preliminary examination we already know that $C_3$ is not IR and that its boundary exhibits an interesting structure.
Studying this example in more detail should lead to a lot more clarity about what is possible for non-IR cones.
\section*{Acknowledgement}
We like to thank Valentin Dannenberg for his very useful insight on our non-IR example and on generalized perfectness in general.
Moreover, we like to thank the anonymous referee for his highly beneficial and thoughtful comments.
Both authors gratefully acknowledge support by the German Research Foundation (DFG)
under grant SCHU 1503/8-1.

\printbibliography[keyword={ref}]
\end{document}